\documentclass[11pt]{article}

\usepackage{amsmath,amsfonts,amsthm,amssymb,graphics,color}

\renewcommand{\Re}{\mathop {\rm Re}\nolimits}

\newtheorem{theorem}{Theorem}
  \newtheorem{lemma}[theorem]{Lemma}

\begin{document}
\title{Rational methods for abstract linear, non-homogeneous problems without order reduction}

\author{C. Arranz-Simón\footnote{IMUVA-Mathematics Research Institute, Facultad de Ciencias, University of Valladolid, 47011 Valladolid, Spain, email: carlos.arranz@uva.es.}  \hspace{2pt}  and C.  Palencia\footnote{IMUVA-Mathematics Research Institute, ETSI de Telecomunicación, University of Valladolid, 47011 Valladolid, Spain, email: cesar.palencia@tel.uva.es.}}

\maketitle

\begin{abstract}
Starting from an A-stable  rational approximation to $\text{e}^z$ of order $p$,
$$r(z)= 1+ z+ \cdots + z^p/ p! + O(z^{p+1}),$$   
families of stable methods are proposed to time discretize abstract  IVP's of the type $u'(t) = A u(t) + f(t)$.
These  numerical procedures turn out to be of order $p$, thus overcoming the order reduction phenomenon, and only one evaluation of $f$ per step is required.
\end{abstract}

\paragraph{2000 Mathematical subject classifications\,:} 65J10 ; 65M20 ; 65M12

\noindent{\bf Keywords\,:}{order reduction, rational methods,Runge-Kutta methods, partial differential equations, abstract evolution equations}

\section{Introduction} We are concerned with the numerical time integration of abstract, linear, non-homogeneous IVPs of the form
\begin{equation}
\label{IVP}
\left\{ \begin{array}{lcl} u'(t) & = & A u(t) + f(t), \quad t \ge 0, \\
u(0)  & = & u_0. 
\end{array} \right.
\end{equation}
It is assumed that (i) $ A: D(A) \subset X \to X$ is the infinitesimal generator of a $\mathcal{C}_0$ semigroup $S_A(t)$, $ t \ge 0$, of linear and bounded operators in a complex Banach space $X$, (ii) $f : [0,+\infty) \to X$ is a continuous source data and  (iii) $u_0 \in X$ is a given initial value. Let us notice that the situation of a source term $f$  defined only in an interval $I  \subset \mathbb{R}$ and an initial time $t_0 =\min I \in I$ can easily be reduced to the format (\ref{IVP}). As it is well known, this  abstract framework covers endless situations of practical interest, including both parabolic and hyperbolic problems (see, e.g., \cite{Hille,Pazy}).

A  Runge-Kutta method of order $p$,  applied to (\ref{IVP}), suffers from the so-called order reduction phenomenon (ORP): it happens that the method, applied in the context of a solution $u \in {\cal C}^{p+1}([0,+\infty),X)$, exhibits an order of convergence $0 \le \nu \le p$ which  is related to the stage order of the method $q \le p$, rather than to $p$ itself. In the context of classical PDEs,  $\nu$ is fractional, no matter how regular the solution $u$ is (in space and time).  In \cite{Ostermann,Ostermann2,IC}, optimal orders of these methods are shown and, for the convenience of the reader,  this is reviewed in Section~3 below.

Some attemps to avoid the ORP are based on the idea of correcting the RK (see, e.g., \cite{Isaias, Mpaz}), but this requires using certain derivatives of the data, something that, in general,  introduces new errors, unless such auxiliary derivatives can be obtained in closed form or can be approximated with great accuracy. Recently, a new, interesting approach has been introduced in \cite{Roberts}, where the authors add new nodes to the RK tableau, leading to the need of adding  some extra evaluations of $f$.

The present paper is based on the remark that  homogeneous problems (\ref{IVP})  (i.e., $f = 0$) can be discretized by using an A-acceptable approximation $r(z)$ to $\text{e}^z$. To this end, it is enough  to implement the {\it rational approximation}, defined by the recurrence
\begin{equation}
\label{recurrence}
u_{n+1}=r(\tau A) u_n, \qquad n \ge 0,
\end{equation}
with initial value $u_0 \in X$ and constant stepsize $\tau > 0$. In this situation there is no order reduction. A step in the recurrence needs solving $s$ linear systems involving $A$, where $s$ is the number of poles (accounted along their multiplicities) of $r(z)$. Furthermore, when $f=0$, an A-stable  RK method applied to (\ref{IVP}) becomes the rational method based on its own stability function.
Notice that in this situation the abcissa $\textbf{c}$ of the RK tableau are not required whatsover.

The main idea is just to cast a non homogeneous IVP into an enlarged, homogeneous problem which is then discretized by a rational method. Essentially, this is achieved by treating $f$ as a new unknown (see Section~4), in the line of the approach used in \cite{Fabrizio} for equations with memory. The resulting discretization is in principle theoretical, but can be implemented within the optimal order just by using auxiliary evaluations of $f$. To this end, some discrete time grid is required and it turns out that sensible choices of such  grids lead to procedures that, per step,  require (i) just a new evaluation of $f$ and (ii) solving a number $s$ of linear systems. Thus, we propose  a  procedure that avoids the ORP. When $r(z)$ is the stability function of a RK method, the new approach maintains the same number $s$ of linear systems per step, as in the RK case, but now only one new evaluation of $f$ is needed.

The paper focuses on the time discretization, though it is clear that full discretizations of PDEs can be treated with standard techniques (see, e.g.,  \cite{Isaias,Mpaz}), and it is organized as follows. In Section~2 we fix the notation and recall the basic results concerning rational approximations, as well as some facts related to interpolation spaces. For the reader's convenience, the order reduction  is explained is Section~3. The new procedure is introduced and studied in Section~4, while Section~5 is devoted to the holomorphic case. Some numerical illustrations are provided in the final Section~6.

\section{Preliminaries and notation}
For a   complex Banach space $(X, \| \cdot \|)$, $M \ge 1$ and $\omega \in \mathbb{R}$, let $\mathcal{G}\left( X, M, \omega \right)$ be the set formed by all the  infinitesimal generators $ A: D(A) \subset X \to X$ of a $\mathcal{C}_0$-semigroup $S_A(t)$, $t \ge 0$, on $X$ with growth
\begin{equation}
\label{semigroupbound}
\| S_A(t) \| \leq \, M \text{e}^{\omega t},  \qquad t\geq 0.
\end{equation}

For the rest of the section, we fix $X$, $M \ge 1$, $\omega \in \mathbb{R}$ and a generator $ A: D(A) \subset X \to X$  in $\mathcal{G}\left( X, M, \omega \right)$. The Hille-Yosida theorem (see, e.g., \cite{Pazy}) guarantees that the spectrum $\sigma(A)$  is contained in the complex half-plane $\Re (z) \leq \omega$ and there hold the inequalities
\begin{equation}
\label{HY}
\| (z I - A)^{-n} \| \le M/(\Re(z) - \omega)^n, \qquad n \ge 1, \quad \Re (z) > \omega.
\end{equation}

An A-acceptable rational mapping $r: \mathbb{C} \to \mathbb{C}$, bounded on the half-plane $\text{Re} z \le 0$, can be developed
into simple fractions
$$
r(z) = r_{\infty} + \sum_{\ell = 1}^k \sum_{j=1}^{m_{\ell}} \dfrac{r_{\ell j}}{(1-z w_{\ell})^j}, \qquad \Re (w_{\ell}) > 0, \quad 1 \le \ell \le k.
$$
Set 
\begin{equation}
\label{umbraltau}
\tau_0(r,\omega) = \left\{ \begin{array}{lcl} +\infty, &\text{if } & \omega \le 0, \\
\min_{1 \le \ell \le k} \Re (1/w_{\ell})/\omega, &\text{if } & \omega > 0, \end{array} \right. 
\end{equation}
so that, in view of (\ref{HY}), for $0 < \tau < \tau\left(r,\omega\right)$,  it makes sense to define the linear, bounded operator in $X$
\begin{equation}
\label{simfrac}
r(\tau A) = r_{\infty}I + \sum_{\ell = 1}^k \sum_{j=1}^{m_{\ell}} r_{\ell j} (I-\tau w_{\ell} A)^{-j}.
\end{equation}

Notice that an evaluation $r(\tau A) v$, $v \in X$, requires solving a total of $s:=\sum_{\ell = 1}^k m_{\ell}$ linear systems of the form
$$
(I-\tau w A) x =y, \qquad \Re(w) >0, \,  y \in X.
$$
The corresponding rational approximation to the solution of the homogeneous, linear problem
\begin{equation}
\label{IVPh}
\left\{ \begin{array}{lcl} u'(t) & = & A u(t), \quad t \ge 0, \\
u(0)  & = & u_0,
\end{array} \right.
\end{equation}
is thus provided by recurrence (\ref{recurrence}). 

Let us also assume that $r(z)$ is an approximation to $\text{e}^z$ of order $p \ge 1$, i.e.,
$$
r(z)- \text{e}^z = O(z^{p+1}), \qquad r(z)- \text{e}^z \ne O(z^{p+2}), \qquad \text{as } z \to 0.
$$   
Next we recall the main results concerning stability and convergence  \cite{BrTh,LTW}.  It  is shown (Theorems~1 and 3 in \cite{BrTh}) that there exist three constants
$$
C_e = C_e(r) >0, \qquad C_s(n) = C_s(r,n) > 0, \qquad \kappa = \kappa (r) \ge 1,
$$
(we emphasize that they do not depend on the considered semigroup whatsoever) such that, for $0 <\tau< \tau_0(r,\omega)$ (notice that for $\omega \le 0$, there is no upper restriction on $\tau$):

(a)  There holds the stability bound
\begin{equation}
\| r^n \left( \tau A \right) \| \leq  M\, C_s(n) \, \text{e}^{\omega^+ \kappa t}, \qquad t_n = n\tau, \qquad \text{with} \quad C_s(n) =O\left(\sqrt{n}\right),
\label{stabilityBT}
\end{equation}
and $\omega^+ = \max\left\lbrace 0, \omega \right\rbrace$. The weak stability (\ref{stabilityBT}) is optimal, in general (it is sharp for $A= \text{d}/\text{d} x$ in maximum-norm \cite{BrTh0}), and it can be improved depending on the behaviour of $r(z)$ (Theorem~2 in \cite{BrTh}) and on the nature of the operator $A$. For instance,  the term $C_s(n)$ becomes $O(1)$ in the following cases: (a1) for  $r(z) = 1/(1-z)$, which corresponds to the implicit Euler method, (a2) when $X$ is a Hilbert space and $A$ is an $\omega$-dissipative operator, and  (a3) when $A$ generates a holomorphic semigroup in a Banach space $X$ (\cite{Crouzeix,Pal1}).

(b) There holds the  optimal convergence estimate (Theorem~3 in \cite{BrTh})
\begin{equation}
\| r^n \left( \tau A \right) u_0 - S_A(n\tau) u_0\| \leq C_e \, M \, t_n \tau^{p'} \, \text{e}^{\omega^+ \kappa t_n} \|A^{p'+1}u_0\|, \qquad n \ge 1,
\label{convergenceBT}
\end{equation}
valid  for $u_0 \in D(A^{p'+1})$, with $1 \le p' \le p$. Observe that $C_s(r,n)$ does not appear in (\ref{convergenceBT}), so that the convergence is optimal even in cases of weak stability. However, stability affects the treatment of the non-homogeneous problems, as well as the analysis of full discretizations. Moreover, it is not hard to conclude that when  the semigroup exhibits a growth 
$$
\| S_A(t) \| \le M(1+t)\text{e}^{\omega^+ t}, \qquad t \ge 0,
$$
then we can modify $C_e$ and get
\begin{equation}
\| r^n \left( \tau A \right) u_0 - S_A(n\tau) u_0\| \leq C_e \, M (1+t_n)\, t_n\, \tau^{p'} \, \text{e}^{\omega^+ \kappa t} \|A^{p'+1}u_0\|, \qquad n \ge 1.
\label{convergenceBT2}
\end{equation}

It is worth noticing that when $|r(\infty)| < 1$ (i.e., $r(z)$ is strongly acceptable) and $A$ generates a holomorphic semigroup, there holds the so-called {\it optimal parabolic} error estimate \cite{LTW}
\begin{equation}
\| r^n \left( \tau A \right) u_0 - S_A(n\tau) u_0\| \leq C_e \, M \tau^{p'} \, \text{e}^{\omega^+ \kappa t} \|A^{p'}u_0\|, \qquad n \ge 1,
\label{convergenceBThol}
\end{equation}
for $u_0 \in D(A^{p'})$, $1 \le p' \le p$. If the semigroup grows as in (\ref{convergenceBT2}) a factor $(1+t_n)$ may be added to the latter. There also holds the {\it bad initial data} error estimate
$$
\| r^n \left( \tau A \right) u_0 - S_A(n\tau) u_0\| \leq C_e \, M n^{-p} \, \text{e}^{\omega^+ \kappa t} \|u_0\|, \qquad n \ge 1,
$$
valid for $u_0 \in X$. Let us point out that, for holomorphic semigroups, the stability and the mentioned error estimates remain valid for A($\theta$)-acceptable rational mappings, as long as the angle $\theta$ is well suited to the holomorphy  angle of the semigroup \cite{Crouzeix,Pal1}.

We finish this section by commenting on some results concerning interpolation spaces which are essential to understand the ORP.
Fix $\omega^* > \omega$ and, for $\nu \ge 0$, set $X_{\nu} = D\left(\left(\omega^* I - A \right)^{\nu}\right)$. The space $X_{\nu}$ is endowed with the graph norm $\|\cdot\|_\nu$ of $\left(\omega^* I - A \right)^{\nu}$. It is well known that $X_{\nu}$ is independent of $\omega^* > \omega$ and that changing $\omega^* > \omega$ results in an equivalent norm.

On the other hand, the real interpolation method \cite{Triebel} provides different intermediate spaces $X_{\nu, p} = [X_0, X_1]_{\nu, p}$, with norms $\| \cdot \|_{\nu,p}$, $0\leq \nu < 1$, $1 \leq p \leq \infty$. It is important to notice that (see Theorem~ 4.17 in \cite{Lunardi})
$$
X_{\nu + \epsilon, p}\hookrightarrow X_{\nu} \hookrightarrow X_{\nu - \epsilon, q} \qquad 0 \leq p,\,  q \leq \infty, \quad 0 \leq \nu - \epsilon < \nu + \epsilon \leq 1,
$$
with continuous embeddings. As a consequence, for  $0 < \nu^* < 1$, $u \in X$, we have that
\begin{equation}
\label{equivinter}
u \notin X_{\nu, p}, \quad \text{for} \quad \nu < \nu^* \leq 1 \, \Leftrightarrow \, u \notin X_{\nu}, \quad \text{for} \quad \nu < \nu^* \leq 1.
\end{equation}

We conclude this section by illustrating the previous concepts in the context of typical evolutionary PDEs in an $L^p$, $p \geq 1$, framework. Let us consider $X = L^p\left(\Omega\right)$, $p \geq 1$, where $\Omega \subset \mathbb{R}^d$ is a bounded domain with regular boundary $\Gamma$. Moreover, we are given two linear partial differential operators $P, \, Q$ on $\Omega$ of orders $m$ and $k \leq m - 1/p$, respectively, with smooth coefficients. The operator $A$ acts on 
$$
D\left(A\right) = \left\lbrace \phi \in W^{m,p}\left(\Omega\right) \,  / \,\,  Q \phi \mid_{\Gamma} = 0 \right\rbrace,
$$
and $A\phi = P\phi$, for $\phi \in D\left(A\right)$. Set $\nu^* =(k+1/p)/m$. Notice that the trace of $Q$, i.e., the operator $\partial : \phi \mapsto Q \phi \mid_{\Gamma}$, can be understood as a linear, bounded operator from $W^{\nu m, p} \left(\Omega\right)$ to $L^p\left(\Omega\right)$, whenever $\nu > \nu^*$ (see, e.g., \cite{Yagi}). The remarkable result in \cite{Lofs} states that, for the $p$-real interpolation method, there holds
\begin{equation}
X_{\nu,p} = [X_0, X_1]_{\nu,p} = \left\{ \begin{array}{lcl} W^{\nu m, p}\left(\Omega\right), &\text{if } & \nu < \nu^*, \\
W^{\nu m, p}\left(\Omega\right) \cap \text{ker} \partial, &\text{if } & \nu > \nu^*. \end{array} \right. 
\label{interpolation}
\end{equation}
This means that, when interpolating, the boundary condition does not need to be imposed when the trace operator does not make sense. Therefore, if the trace of a smooth mapping $\phi : \Omega \rightarrow \mathbb{C}$ is different from $0$ on $\Gamma$, then $\phi$ cannot belong to $X_{\nu,p}$, for $\nu > \nu^*$. Notice that this is also true for the domains of the fractional powers $X_{\nu}$ according to (\ref{equivinter}). This fact is what governs order reduction.

\section{Order reduction for Runge-Kutta  methods}
In this section we briefly review the basic results concerning the stability and convergence of RK methods applied to (\ref{IVP}). Let
\begin{equation*}
\begin{array}{c|c}
\textbf{c} & W\\
\hline\\[-10pt]
&  \textbf{b}^T
\end{array}, \quad\qquad \mathbf{b}, \mathbf{c} \in \mathbf{R}^s, W \in \mathbf{R}^{s \times s},
\end{equation*}
be the Butcher tableau of a given RK method of order $p \ge 1$. In our context, it is natural to assume that the stability function 
$$
r(z) = 1 + z \mathbf{b}^T\left(I - z W \right)^{-1} \mathbf{e}, \qquad \mathbf{e} = [1, \dots, 1]^T \in \mathbb{R}^s,
$$
is A-stable, i.e,
\begin{equation}
|r(z)| \leq 1, \qquad \Re (z) \leq 0.
\label{Astable}
\end{equation}

As we have already commented in the Introducion, for non-homogeneous problems (\ref{IVP}) the order $p$ of convergence is not achieved in general, no matter the time-regularity of its solution. Actually, let  $u \in \mathcal{C}^{p+1} \left([0, \infty), X \right)$ be the solution of (\ref{IVP}). The local error of the RK method applied to this problem is given by the expression \cite{IC,Ostermann,Ostermann2}
\begin{equation}
\label{localerrorRK}
\epsilon_n = \sum_{l = q+1}^p \tau^l r_l\left(\tau A\right) u^{(l)}(t_n) + O \left(\tau^{p+1}\right),
\end{equation}
where
\begin{equation}
r_l(z) = z \mathbf{b}^T \left(1- zW\right)^{-1} \left(\mathbf{c}^l - l \mathbf{c}^{l-1}\right), \qquad \text{for} \, q+1 \leq l \leq p.
\end{equation}
Moreover, the order conditions of the method guarantee that
\begin{equation}
r_l(z) = z^{p-q} r^*_l(z), \qquad \text{for} \, q+1 \leq l \leq p,
\end{equation}
for certain rational mappings $r^*_l(z)$, $q+1 \leq l \leq p$. 
Notice that since $r_l(z)$ and $r_l^*(z)$, $q+1 \le l \le p$, possess no poles on the half-plane $\Re (z) \le 0$, we can argue as in  (\ref{simfrac}) and see that  the operators $r_l(\tau A)$ and $r_l^*(\tau A)$, $q+1 \leq l \leq p$,  are bounded for $0 < \tau < \tau_0\left(r, \omega\right)$. 

For  $x \in X_{p-q}$ we also have (Lemma~ 4 in \cite{BrTh}), that
$$
\| r_l(\tau A) x\| \leq C_e\, M\, \tau^{p-q} \|x\|_{p-q}, \quad q+1 \leq l \leq p,
$$
and, by interpolation, we deduce that for $x \in X_\nu$, $0 \leq \nu \leq p-q$,
$$
\| r_l(\tau A) x\| \leq C_e\, M\, \tau^\nu \|x\|_\nu, \quad q+1 \leq l \leq p.
$$
Thus, in view of (\ref{localerrorRK}), for $u \in \mathcal{C}^{p+1} \left([0, \infty), X_\nu \right)$, we get
$$
\| \epsilon_n \| = O(\tau^{q+1+\nu}) \sup_{0 \le t \le n\tau} \| u^{(p+1)}(t)\|_\nu.
$$
Therefore, only under the stronger assumption $u \in \mathcal{C}^{p+1} \left([0, \infty), X^{p-q} \right)$ we reach the optimal local order $p+1$.  These ideas can be extended to the situation of variable step-sizes. Besides, for constant step-sizes and 
in case $r(\infty) \ne 1$, the clever summation-by-parts argument in \cite{Ostermann,Ostermann2} (extended in \cite{Isaias} to general semigroups) leads to the improvement
$$
\| \epsilon_n \| = O(\tau^{q+1+\mu}) \sup_{0 \le t \le n\tau} \| u^{p+1}(t)\|_\nu,
$$
where $\mu = \min(\nu+1,p-q)$.

The origin of the ORP relies in the fact that, as we mentioned in (\ref{interpolation}), all we can expect in the context of standard PDEs is that 
$u \in \mathcal{C}^{p+1} \left([0, \infty), X_{\nu^*} \right)$, for some well defined value 
$0 < \nu^*<1$. Setting $\nu = \nu^*$ (or $\nu = \min(\nu^*+1,p-q)$ when $r(\infty) \ne 1$), we easily get the error estimate
$$
\| u(t_n) - u_n \| = C_s(r,n) \tau^{q+\nu}, \qquad n \ge 1,
$$
where $C_s(r,n)$ stands for the stability bound of the recurrence. It is worth noticing that such a  fractional order of convergence is the one occurring in practical computation.

\section{The method}
Our starting remark is that, in view of (\ref{convergenceBT}), a RK method applied to a homogeneous problem results in a rational method (\ref{recurrence}) and exhibits its optimal (classical) order $p$. This suggests to cast (\ref{IVP}) into an extended homogeneous problem, which is then time discretized by a rational method.

To this end, we introduce the space $Y = \mathcal{C}_{ub} \left([0,\infty), X \right)$ of all bounded, uniformly continuous functions from $[0,\infty)$ to $X$, endowed with the supremum norm. We consider the \textit{semigroup of translations} on this space, whose properties are summarized in the following lemma, that we state without proof (see \cite{Engel} for the case $X = \mathbb{C}$).

\begin{lemma} Let $X$ be a Banach space and $Y = \mathcal{C}_{ub}\left([0,\infty), X\right)$. Then the semigroup of translations $S_B(t): Y \rightarrow Y$, $t \geq 0$, defined by
\begin{equation*}
\left[ S_B(t)\, v \right] (s) = v(t+s), \qquad v \in Y, \, s \geq 0,
\end{equation*}
is a $\mathcal{C}_0$ semigroup. Its infinitesimal generator $B: D\left(B\right) \subset Y \rightarrow Y$ belongs to $\mathcal{G}\left(Y,1,0\right)$, $D\left(B\right) = \left\lbrace v \in Y / v' \in Y \right\rbrace$ and $B v = v'$ for $v \in D\left(B\right)$.
\label{translations}
\end{lemma}

We consider $X$ and $A : D(A) \subset X \rightarrow X$ as in (\ref{IVP}), $Y = \mathcal{C}_{ub} \left([0,\infty), X \right)$ and $B$ as in the previous lemma. We define the bounded, linear operator $L : Y \rightarrow X $ by $L y = y(0)$, for $y \in Y$, the Banach product space $Z = X \times Y $ with the norm $\|(x,y)\|_Z = \|x\|_X + \|y\|_Y$ and the operator
\begin{equation}
G\begin{pmatrix}
u \\ 
v
\end{pmatrix} = 
\begin{pmatrix}
A & L \\ 
0 & B
\end{pmatrix} 
\begin{pmatrix}
u \\ 
v
\end{pmatrix},
\quad \text{for} \,\, (u, v)^T \in D\left(G \right):=D\left(A \right) \times D\left(B \right).
\label{HomGenerator}
\end{equation}
A direct comprobation shows that $G$ is the infinitesimal generator of a semigroup in the space $Z$, with growth
$$
\| S_G(t)\| \le M(1+t)\text{e}^{\omega^+ t}, \qquad t \ge 0.
$$

Notice that for $u_0 \in D\left(A\right)$, $f \in D\left(B\right)$, the solution of the IVP
\begin{equation}
\label{IVPe}
\left\{ \begin{array}{lcl} U'(t) & = & G \, U(t), \quad t \geq 0, \\
U(t_0)  & = & (u_0,f)^T,
\end{array} \right.
\end{equation}
is given by $U(t) = (u(t),v(t))^T$, $t \geq 0$, where
\begin{equation*}
u'(t) = A\, u(t) + L \, v(t), \qquad v'(t) = B\, v(t).
\end{equation*}
Taking into account the definition of $S_B$ and $L$, 
\begin{equation*}
v(t) = f(t+\cdot), \qquad L\, v(t) = f(t),
\end{equation*}
we obtain
\begin{equation}
\label{IVP2}
\left\{ \begin{array}{lcl} u'(t) & = & A u(t) + f(t), \quad t \geq 0, \\
u(t_0)  & = & u_0. 
\end{array} \right.
\end{equation}
By continuity, we conclude that even the generalised solutions of (\ref{IVP2}) are provided by the first component of the solutions of (\ref{IVPe}). Thus, it is natural to approximate (\ref{IVP}) by means of a rational method applied to (\ref{IVPe}). In this way it is clear by (\ref{convergenceBT}) that no order reduction occurs. This leads to the recurrence
\begin{equation}
\bar{U}_{n+1} = r\left(\tau G\right) \begin{pmatrix}
\bar{u}_n \\ 
\bar{v}_n
\end{pmatrix}  =r_\infty \begin{pmatrix}
\bar{u}_n \\ 
\bar{v}_n
\end{pmatrix} + \sum_{\ell = 1}^k\sum_{j=1}^{m_{\ell}} r_{\ell  j} \left(I - \tau w_{\ell} G\right)^{-j} \begin{pmatrix}
\bar{u}_n \\ 
\bar{v}_n
\end{pmatrix}.
\label{metodoexacto}
\end{equation}
Thus, assuming that $u, f \in \mathcal{C}^{p+1}\left([0,\infty),X\right)$, it is clear that $U\in \mathcal{C}^{p+1}\left([0,\infty),Z\right)$, and since (\ref{IVPe}) is a homogeneous problem, $U \in \mathcal{C}\left([0,\infty), D\left(G^{p+1}\right)\right)$. Notice that $D\left(G^{p+1}\right)$ may be different from $D\left(A^{p+1}\right) \times D\left(B^{p+1}\right)$, so $u$ may not lie in $D\left(A^{p+1}\right)$. Under this assumption, the convergence result (\ref{convergenceBT2}) applied to $G$ and initial data $u_0$ and $v_0 = f$ guarantees that
\begin{equation}
\|u(t_n) - \bar{u}_{n} \| \leq C_eM(1+ t_n)\,t_n \,  \tau^p \text{e}^{\omega^+ \kappa t_n}  \left(\|u^{(p+1)}\|_{\infty} + \|f^{(p+1)}\|_{\infty}\right), \quad n \geq 1,
\label{exactconvergence}
\end{equation}
whereas the same result with generator $B$ and initial data $v_0$ leads to
\begin{equation}
\|f(t_n + \cdot) - \bar{v}_{n} \|_{\infty} \leq C_e t_n  \tau^p    \|f^{(p+1)}\|_{\infty}, \quad n \geq 1.
\label{translationconvergence}
\end{equation}
The practical difficulty of (\ref{metodoexacto}) lies in the fact that it is not possible to implement exactly the rational method by using evaluations of $f$ on a discrete mesh. The method we propose provides approximations $u_n$ to $\bar{u}_n$ that avoid this drawback.
 
We start by computing the resolvent of $G$ (and its powers) in terms of the resolvents of $A$ and $B$.
\begin{lemma} The k-th powers of the resolvent of the operator $G$ in (\ref{HomGenerator}) can be expressed by computing $k$ separate resolvents of the operators $A$ and $B$. In fact, if $(u_0,v_0)^T \in X\times Y$, we can obtain $U_k = (u_k, v_k)^T= (\lambda I-G)^{-k} (u_0, v_0)^T$ via the following recurrence
\begin{equation}
v_j = \left( \lambda I - B \right)^{-1} v_{j-1}, \qquad u_j = \left(\lambda I - A \right)^{-1} \left(u_{j-1} + L v_j\right), \quad for \, j=1, \dots, k.
\end{equation}
\label{resolventG}
\end{lemma}
\begin{proof}
The identity $U_k = (\lambda I-G)^{-1}U_{k-1}$ leads to the equation
\begin{equation*}
\begin{pmatrix}
\lambda I - A & - L \\ 
0 &\lambda  I - B
\end{pmatrix} 
\begin{pmatrix}
u_k \\ 
v_k
\end{pmatrix}
= 
\begin{pmatrix}
u_{k-1} \\ 
v_{k-1}
\end{pmatrix},
\end{equation*}
that can be immediately expressed as
\begin{equation*}
v_k = (\lambda I-B)^{-1} v_{k-1} = (\lambda I-B)^{-k} v_{0},
\end{equation*}
and
\begin{equation}
u_k = (\lambda I-A)^{-1} \left(u_{k-1} + L v_k\right) = (\lambda I-A)^{-1} \left(u_{k-1} + L (\lambda I-B)^{-k} v_{0}\right).
\label{discreterecurrence}
\end{equation}\end{proof}

By using (\ref{discreterecurrence}) with $\lambda = 1/\tau w_{\ell}$ and the discrete variation-of-constants formula, the two components of (\ref{metodoexacto}) can be written as
\begin{equation}
\quad \bar{u}_{n+1} = r\left(\tau A \right) \bar{u}_n + \tau E(\tau)\bar{v}_n, \qquad \bar{v}_{n+1} = r\left(\tau B \right) \bar{v}_n, \quad n \geq 1,
\label{metodoexactoshort}
\end{equation}
where $E(\tau): Y \rightarrow X$ is the linear operator given by
\begin{equation}
E(\tau)v = \sum_{\ell = 1}^k \sum_{j=1}^{m_{\ell}} r_{\ell, j} w_{\ell} \sum_{i=1}^j \left(I - \tau w_{\ell} A \right)^{-j+i-1} L \left(I - \tau w_{\ell} B\right)^{-i} v,
\label{operatorE}
\end{equation}
for $v\in Y$, which is bounded for $0 < \tau < \tau_0$.

As we mentioned, even though the semigroup $S_B(t)$ is trivial, the resolvents of $B$ cannot be computed in a direct way by using evaluations of $f$ along a discrete mesh. However, for our purpose, it will be sufficient to approximate the resolvent with a suitable order. Recalling that 
\begin{equation*}
\left(\lambda I - B \right)^{-1}   = \int_0^{\infty} \text{e}^{-\lambda s}  S_B(s) \, ds ,
\end{equation*}
we see that
\begin{equation}
\left[\left(\lambda I - B \right)^{-1} v\right](t) = \int_0^{\infty}  \text{e}^{-\lambda s} v\left(t + s\right)\, ds,
\label{resolventBintegral}
\end{equation}
for $t\geq 0$ and $v \in Y$. Though in principle it makes sense to approximate (\ref{resolventBintegral}) by some adequate quadrature formula, our approach is based on Lemma~\ref{lemaVandermonde} below that also allows us to approximate $F(\tau B)$ for more general functions $F$, in particular, for the powers of the resolvent of $B$.  

In what follows, for $v \in Y$, $\boldsymbol{c} \in \mathbb{R}^p$ and $t, \, \tau >0$ such that  $t+\tau \boldsymbol{c} \ge 0$, $v(t + \tau\boldsymbol{c})$ denotes $\left[v(t + \tau c_1), \dots, v(t + \tau c_p)\right]^T$. 

\begin{lemma} Let $F$ be a rational mapping with no poles on the half-plane $\Re (z) \leq 0$ and $\boldsymbol{c} \in \mathbb{R}^p$ with  $c_k \neq c_j$ for $k \neq l$. Then, there exist a unique $\boldsymbol{\gamma} =\{ \gamma_k\}_{k=1}^p \in \mathbb{R}^p$ and $ C > 0$ such that for  $t, \, \tau >0$,  $t+\tau \boldsymbol{c} \ge 0$,
\begin{equation}
\| F\left(\tau B\right)v(t+\cdot) - \boldsymbol{\gamma}^T \cdot v(t+\tau \boldsymbol{c}+\cdot)\| \leq C \tau^p \|B^p v\|,
\qquad v \in D(B^p)
\label{resolventB}
\end{equation}

\label{lemaVandermonde}
\end{lemma}
\begin{proof}
Le us consider the Taylor expansions 
\begin{eqnarray*}
F(z) & = & f_0 + f_1 z + \cdots + f_{p-1}z^{p-1} + O\left(z^{p}\right), \\
\text{e}^{c_k z} & = & 1 + c_k z + \cdots +  \dfrac{c_k^{p-1}}{(p-1)!} z^{p-1}+ O\left(z^{p}\right), \quad 1 \leq k \leq p,
\end{eqnarray*}
and try to find $\boldsymbol{\gamma} =\{ \gamma_k\}_{k=1}^p$ such that
$$
H(z) = F(z) - \sum_{k=1}^p \gamma_k \text{e}^{c_k z} = O\left(z^{p}\right).
$$
This leads to the  Vandermonde system
\begin{equation}
c_1^j \gamma_1 + \cdots + c_p^j \gamma_p =j! \,F_j, \qquad 0 \leq j \leq p-1,
\label{Vandermonde}
\end{equation}
which, since $c_k \neq c_l$, for $k \neq l$, has a unique solution. 

Set $P_{p-1}(z) = \sum_{k=0}^{p=1} f_k z^k$.
For $v \in D(B^p)$, the above calculation implies that, for  $t>0$, $\tau >0$, with $t+\tau c_k \ge 0$, $1 \le k \le p$,
$$
\sum_{k=0}^{p-1} f_k\tau^k v^{(j)}(t) = \sum_{k=0}^{p-1} \gamma_k \tau^k v(t+\tau c_k)+\tau^p O(B^pv),
$$ 
where the involved constant in $O(B^pv)$ is independent of $\tau$. Since this is also valid for $t+s$, with $s \ge 0$, we can claim that there exists $C_1 >0$ (independent of $t$ and $\tau$ as indicated) such that
\begin{equation}
\label{cotaP}
\| P_{p-1}(\tau B) S_B(t)v- \boldsymbol{\gamma}^T \cdot v(t+\tau c_k+\cdot) \| \le C_1\tau^p \| B^p v \|, .
\end{equation}
Let us now consider the rational mapping $F_0(z) = (F(z) - P_{p-1}(z))/z^p$, that is bounded on the half plane $\Re(z) \le 0$. Given  that  $(F-F(\infty)), \, F', \,F_0$ and $F_0' $ are square integrable along the imaginary axis, Lemma~1 and Lemma~2 in \cite{BrTh} imply that   $F(z)$ and $F_0(z)$ are Laplace-Stieljes transforms of originals measures of bounded variation, with support contained in $[0,+\infty)$.  Moreover, by Lemma~4 in \cite{BrTh}, we have
$$
F(\tau B)v- P_{p-1}(\tau B)w = \tau^p F_0(\tau B)  B^pw, \qquad w \in D(B^p),
$$
and, by Lemma~5 in \cite{BrTh},  the norms $\| F_0(\tau B) \|$ are uniformly bounded, for $\tau >0$, by the total variation of the original measure of $F_0$. Therefore, there exists $C_2>0$ such that
$$
\| F(\tau B)w -P_{p-1}(\tau B) w\| \le C_2\tau^p \| B^p w\|, \qquad w \in D(B^p),
$$
an estimate that,  applied to  $w=S_B(t)v$, leads to
\begin{equation}
\label{cotaF}
\| F(\tau B)S_B(t)v -P_{p-1}(\tau B) S_B(t)v\| \le C_2\tau^p  \| B^p v\|.
\end{equation}
To conclude the proof, we just write
$$
F(\tau B)S_B(t)v-\boldsymbol{\gamma}^T \cdot v(t+\tau \boldsymbol{c})=(I) + (II),
$$
where
\begin{eqnarray*}
(I) &=& F(\tau B)S_B(t)v - P_{p-1}(\tau B) S_B(t) v, \\
(II) &= &P_{p-1}(\tau B) S_B(t) v - \boldsymbol{\gamma}^T \cdot v(t+\tau \boldsymbol{c}),
\end{eqnarray*}
and recall (\ref{cotaF}) and (\ref{cotaP}).
\end{proof}

Notice that the proof also shows a way to find the adequate $\boldsymbol{\gamma}$ by solving the Vandermonde system (\ref{Vandermonde}) and that $\boldsymbol{\gamma}$ depends only on $F$ and $\boldsymbol{c}$, but not on $v, \,\tau,\,t$.

We finally propose the following method. Set $\mathcal{D}= \left\lbrace \boldsymbol{c} \in \mathbb{R}^p / c_i \neq c_j, i \neq j \right\rbrace$ and choose a sequence $\boldsymbol{c_n}$, $n\geq 1$, in $\mathcal{D}$. Lemma \ref{lemaVandermonde} applied to $F_{\ell,j}(z) = (1-w_{\ell} z)^{-j}$, $1 \leq \ell \leq k$, $1\leq i \leq m_{\ell}$, and a vector $\boldsymbol{c_n}$, provides $\boldsymbol{\gamma_{\ell,i}^n} \in \mathbb{R}^p$ that leads to an approximation of order $p$ (in the sense of (\ref{resolventB}))
\begin{equation}
L \left(I - \tau w_{\ell} B \right)^{-i} v \approx \boldsymbol{\gamma_{\ell,j}^n}^T \cdot v\left(\tau \boldsymbol{c_n}\right).
\label{resolventBaprox}
\end{equation}
We then adopt
\begin{equation}
u_{n+1} = r\left(\tau A\right) u_n + \tau E_n(\tau)f\left(t_n + \tau \boldsymbol{c_n}\right),
\label{metodoaproxshort}
\end{equation}
where $E_n(\tau) : X^p \rightarrow X$, $0 < \tau < \tau_0$, is the bounded, linear operator given by
\begin{equation}
E_n(\tau)f\left(t_n + \tau \boldsymbol{c_n}\right) = \sum_{\ell = 1}^k \sum_{j=1}^{m_{\ell}} r_{\ell, j} w_{\ell} \sum_{i=1}^j \left(I - \tau w_{\ell} A \right)^{-j+i-1} \boldsymbol{\gamma^n_{\ell,i}}^T \cdot f\left(t_n + \tau \boldsymbol{c_n}\right).
\label{operatorEn}
\end{equation}
Recalling (\ref{operatorE}) and in view of (\ref{resolventBaprox}) there holds
\begin{equation}
\|E_n(\tau)f(t_n + \tau \boldsymbol{c_n}) - E(\tau)f(t_n + \cdot)\| \leq K_n \tau^p \|f^{(p)}\|_{\infty},
\label{operatorEaprox}
\end{equation}
for some $K_n = K_n(\boldsymbol{c_n}) > 0$. One step in (\ref{metodoaproxshort}) requires solving $s$ resolvents of $A$, as in the homogeneous case. Moreover, for arbitrary $\boldsymbol{c_n}$, it also requires $p$ evaluations of the function $f$. However, the vectors $\boldsymbol{c_n}$ can be chosen in such a way that only one evaluation per step is done for $n \geq 2$ (see Section 6 for details). To prove convergence we also require that all the $\boldsymbol{c_n}$ lie in a compact set $\mathcal{K} \subset D$, in such a way that $K_n(\boldsymbol{c_n}) \leq K$ when $n\geq 1$.

Then we state the main result of the paper, that assures that the method (\ref{metodoaproxshort}) converges to the solution of (\ref{IVP}) without order reduction. Anyway,  the optimal order $p$ could be reduced in case of weak stability (\ref{stabilityBT}).

\begin{theorem}
Let $u : [0, \infty) \rightarrow X$ be the solution of (\ref{IVP}) to be approximated in the interval $[0, T]$ with constant step-size $0 < \tau = T/N < \tau_0$. Assume that $u \in \mathcal{C}^{p+1} \left([0,\infty), X\right)$, $f \in \mathcal{C}^{p+1} \left([0,\infty), X\right)$. Let $u_n$ be the numerical approximation to $u(t_n)$ obtained by the modified rational method (\ref{metodoaproxshort}) with nodes $\boldsymbol{c_n}$ in some compact set $\mathcal{K} \subset \mathcal{D}$. Then, there exists a constant $K =K(\mathcal{K}) > 0$ such that, for $0 \leq n \leq N$,
$$
\|u(t_n) - u_n\|  \leq  K C_e C_s(n) (1+t_n)t_n M\text{e}^{\omega^+\kappa t_n}  \tau^{p} \left(\|u^{(p+1)}\|_{\infty} + \|f^{(p)}\|_{\infty} + \|f^{(p+1)}\|_{\infty}\right).
$$\label{mainth}
\end{theorem}
\begin{proof} First, notice that (\ref{metodoexactoshort}) can be written as
\begin{equation*}
\bar{u}_{n+1} = r\left(\tau A\right)\bar{u}_{n} + \tau E(\tau)\, r^{n}\left(\tau B \right) f,
\end{equation*}
and subtracting this expression from (\ref{metodoaproxshort}),
\begin{eqnarray*}
u_{n+1} - \bar{u}_{n+1} & = & r\left(\tau A\right) \left(u_n - \bar{u}_n\right) + \tau \left(E_n(\tau)f(t_n+ \tau \boldsymbol{c_n}) - E(\tau) r^{n}\left(\tau B\right)f\right) \\
& = & r\left(\tau A\right) \left(u_n - \bar{u}_n\right) + \tau \left(E_n(\tau)f(t_n+ \tau \boldsymbol{c_n}) - E(\tau)f(t_n + \cdot)\right)\\
&  & + \,\,\, \tau \left(E(\tau) \left(f(t_n + \cdot) - r^{n}\left(\tau B\right)f \right)\right),
\end{eqnarray*}
with $u_0 = \bar{u}_0$. Then, by the variation-of-constants formula, the error can be bounded by three terms
\begin{equation}
\|u(t_n) - u_n \| \leq (I) + (II) + (III),
\label{localerror}
\end{equation}
where
\begin{eqnarray*}
(I) & = &   \|u(t_n) - \bar{u}_n\|, \\
(II) & = & \tau \sum_{k=0}^{n-1} \| r^{n-k}\left(\tau A\right)\| \| E_k(\tau)f(t_k + \tau \boldsymbol{c_k}) - E(\tau)f(t_k + \cdot)\|, \\
(III) & = & \tau \sum_{k=0}^{n-1} \| r^{n-k}\left(\tau A\right) \| \| E(\tau) \left(f(t_k + \cdot) - r^{k}\left(\tau B\right)f \right)\|.
\end{eqnarray*}
To bound these terms, we proceed as follows. A bound for (I) is given by (\ref{exactconvergence}). Moreover, taking into account the compactness of $\mathcal{K}$, (\ref{stabilityBT}) and (\ref{operatorEaprox}) we get
\begin{equation*}
(II) \leq K\, M \, C_s(n)\, t_n \, \text{e}^{\omega^+\kappa t_n}\, \tau^p \|f^{(p)}\|_{\infty}, \qquad K=K(\mathcal{K}).
\end{equation*}
Finally, the fact that $E(\tau)$ is bounded together with (\ref{stabilityBT}) and (\ref{translationconvergence}) assures that
\begin{equation*}
(III) \leq M \, C_e \, C_s(n)\,t_n^2\,\text{e}^{\omega^+\kappa t_n}\, \tau^p \|f^{(p+1)}\|_{\infty},
\end{equation*}
and the proof concludes combining the three estimates.
\end{proof}

\section{The holomorphic case}
When both the semigroup $S_A(t)$ and the source term $f(t)$ admit holomorphic continuations to some sector $$\Sigma_{\theta} = \left\lbrace z \in \mathbb{C} \, / \, |\text{arg}(z)| \leq \theta \right\rbrace, \qquad 0 < \theta < \pi/2, $$ it is possible to reformulate Theorem 4.4 in the line of (\ref{convergenceBThol}). Moreover, as we mentioned, the stability constant $C_s(n)$ can be dropped so the order $p$ is actually recovered  and further results concerning bad initial values and a variable step-size version of the method can be considered.

We introduce the space $Y_{\theta} = \mathcal{A}_{ub} \left(\Sigma_{\theta}, X \right)$ of all bounded, uniformly continuous, analytic functions from $\Sigma_{\theta}$ to $X$ endowed with the supremum norm, that result to be a Banach space. There we define the holomorphic semigroup of translations $S_{B_{\theta}} : Y_{\theta} \rightarrow Y_{\theta}$, $ z \in \Sigma_{\theta}$, given by
$$
\left[ S_{B_{\theta}}(z)\, v \right](t) = v(z + t), \quad v \in Y_{\theta}, t \in \Sigma_{\theta}.
$$
Defining $D\left(B_{\theta}\right) = \left\lbrace v \in Y_{\theta} / v' \in Y_{\theta} \right\rbrace$, the corresponding generator is $B_{\theta}: D\left(B_{\theta}\right) \subset Y_{\theta} \rightarrow Y_{\theta}$  with $B_{\theta}v = v'$ for $v \in D\left(B_{\theta}\right)$. Then, $G_{\theta} = X \times Y_{\theta}$ becomes holomorphic too (see, e.g., section 2.3 in \cite{Yagi}) and we can state the holomorphic version of Theorem \ref{mainth}.

\begin{theorem}
Let $u : [0, \infty) \rightarrow X$ be the solution of (\ref{IVP}) to be approximated on the interval $[0, T]$ with constant step-size $0 < \tau = T/N < \tau_0$. Assume that $A$ generates a holomorphic semigroup with holomorphy angle $0 < \theta < \pi/2$ and that $r$ is a strongly $A(\theta)$-acceptable mapping. Assume also that $u \in \mathcal{C}^{p} \left([0,T], X\right)$, $f \in \mathcal{A}_{ub} \left(\Sigma_{\theta}, X\right)$. Let $u_n$ be the numerical approximation to $u(t_n)$ obtained by the modified rational method (\ref{metodoaproxshort}) with nodes $\boldsymbol{c_n}$ in some compact set $\mathcal{K} \subset \mathcal{D}$. Then, there exists a constant $K =K(\mathcal{K})> 0$ such that
$$
\|u(t_n) - u_n\| \leq K \, M\, C_e \left(1 + t_n \right)\, \text{e}^{\omega^+\kappa t_n} \, \tau^{p} \left(\|u^{(p)}\|_{\infty} + \|f^{(p)}\|_{\infty}\right), \quad 0 \leq n \leq N.
$$
\label{holmainth}
\end{theorem}
\begin{proof}
The proof is similar to that of Theorem \ref{mainth}, but in this case we take advantage of the optimal parabolic estimate (\ref{convergenceBThol}) instead of (\ref{convergenceBT}). We split the local error in the same three terms as in (\ref{localerror}). Then, the estimate (\ref{convergenceBThol}) applied to $G_{\theta}$ and initial data $u_0$ and $f$ leads to 
\begin{equation*}
(I) \leq M \, C_e  \left(1 + t_n \right) \,\text{e}^{\omega^+ \kappa t_n} \tau^p \, \left(\|u^{(p)}\|_{\infty} + \|f^{(p)}\|_{\infty}\right).
\end{equation*}
Moreover, as we mentioned in Section 3, in the holomorphic case $C_s(n)$ is $O(1)$, so that 
\begin{equation*}
(II) \leq K \, C_s\,M \, t_n \,\text{e}^{\omega^+ \kappa t_n} \, \tau^p \, \|f^{(p)}\|_{\infty}, \qquad K=K(\mathcal{K}).
\end{equation*}
To conclude, the application of (\ref{convergenceBThol}) to $B_{\theta}$ with initial data $f$ and (\ref{stabilityBT}) gives
\begin{equation*}
(III) \leq M \, C_s\, C_e\, t_n  \text{e}^{\omega^+ \kappa t_n} \, \tau^p \|f^{(p)}\|_{\infty},
\end{equation*}
thus completing the proof.
\end{proof}

Finally, let us comment that the previous theorem can be adapted to cover the situation of $A(\theta)$-acceptable rational mappings and variable step sizes \cite{Pal3}.

\section{Numerical illustrations}
In this section we present several numerical illustrations to show the convergence behaviour of the proposed method. We deal with simple PDEs which are integrated by the method of lines. The space discretization is accomplished by standard finite differences. If $h > 0$ stands for the space-discretization parameter, we are lead to systems of ODEs 
\begin{equation}
\label{IVPsdis}
\left\{ \begin{array}{lcl} u_h'(t) & = & A_h u_h(t) + f_h(t), \quad t \geq 0, \\
u_h(0)  & = & u_{0,h}. 
\end{array} \right.
\end{equation}
To focus on the error due to the time integration we proceed as follows:
\begin{enumerate}
\item We start from a known solution $u(t,x)$ of (\ref{IVP}), corresponding to some source term $f$, so that $u$ takes values in the intermediate space $X_{\nu}$, with $\nu < \nu^*$ ($\nu^*$ given in Section 3).
\item We adjust $f_h$ in such a way that the restriction of $u$ to the discrete mesh is the exact solution of (\ref{IVPsdis}).
\item We accept as reasonable that the order reduction of the RK method applied to (\ref{IVPsdis}) is close to the one determined by $\nu^*$ \cite{Mpaz}.
\item We compare a RK method with its rational version (\ref{metodoaproxshort}).
\end{enumerate}

The illustrations use a constant step size $\tau > 0$ and the auxiliar vectors $\boldsymbol{c_n}$, $n\geq 0$, are chosen to be
\begin{equation}
\boldsymbol{c_n} = \left\{ \begin{array}{l} \left[-n,-n+1,\dots, p-1-n \right] \quad \text{for} \quad n = 1, \dots, p-1 \\
\left[-p+1,-p+2,\dots,0\right] \quad \quad \quad \text{for} \quad n > p-1
\end{array} \right.
\label{dotgrid}
\end{equation}
in such a way that only one function evaluation per step is required for $n \geq 2$.

Concerning the implementation, first of all, we need linear solvers for the involved systems
$$
(I-\tau w_l \, A_h)x_h = y_h \in X_h, \, \, 1 \le l \le k.
$$ 

To get the  $s:=\sum_{\ell=1}^k m_l$ vectors $\boldsymbol{\gamma^n_{\ell,j}} \in \mathbb{R}^p$, $1 \le \ell \le k$, $1 \le j \le m_j$, required in (\ref{operatorEn}), we consider the  expansions 
$$
(1-w_ljz)^{-j} = \sum_{q=0}^{p-1} \binom{-j}{q} w_l^q z^q + O(z^p)
$$
and solve the $s$ corresponding  Vandermonde systems (\ref{Vandermonde}), for $F(z) = (1-w_lz)^{-j}$, $1 \le k$, $1 \le j \le m_j$ and ${\bf c_n}$, $n \ge 1$. It is clear that, in our context, we can restrict the task to solve such systems for $1 \le n \le p$. We must notice that, for relatively large values of $p$, the Lebesgue constant of the equispaced nodes ${\bf c_n}$, $n \ge 1$, is large. In these situations it is much better to consider other families of nodes, for instance the ones relying on Chebyshev nodes. Finally, once we know the  coefficients, (\ref{operatorEn}) is implemented in a Horner like algorithm, by using the available linear solvers.

For instance, the 3-stages Gauss method (see, e.g., \cite{Hairer}) has a rational stability function that can be written in the form
$$
r(z) = \dfrac{r_1}{1- w_1 z}  + \dfrac{r_2}{1- w_2 z} + \dfrac{r_3}{1- w_3 z} - 1,
$$
with $r_{\ell}, w_{\ell} \in \mathbb{R}$, $w_{\ell} > 0$, for $\ell = 1, 2, 3$. As the method has order $p=6$, in each step we fix ${\bf c_n} \in \mathbb{R}^6$ (see (\ref{dotgrid}) for the examples) and obtain three vectors $\boldsymbol{\gamma^n_{\ell}} \in \mathbb{R}^6$ by solving the mentioned linear systems. Then, a step of the method is done by computing
$$
u_{n+1} = \sum_{\ell =1}^3 r_{\ell} \left(I - \tau w_{\ell} A \right)^{-1} \left(u_n + \, \boldsymbol{\gamma^n_{\ell}} \cdot f\left(t_n + \tau \boldsymbol{c_n}\right)\right) - u_n.
$$
\paragraph*{One-dimensional hyperbolic problem} We consider a classic hyperbolic problem in the unit interval with homogeneous boundary conditions,
\begin{equation}
\label{Example3}
\left\{ \begin{array}{lcll} u_t(t,x) & = & -  u_{x}(t,x) + f(t,x), & \quad 0\le t \le 1, \quad 0 \leq x \leq 1, \\
u(0,x)  & = & u_0(x), & \quad 0 \leq x \leq 1, \\
u(t,0) & = & 0, & \quad 0 \leq t \leq 1,
\end{array} \right.
\end{equation}
where $f : [0,1] \times [0,1] \rightarrow \mathbb{C}$, $u_0 : [0,1] \rightarrow \mathbb{C}$. In order to fit the problem in our framework, we take $X = L^2[0,1]$, $A = -d/dx$, $D(A) = \left\lbrace u \in H^1[0,1] : u(0+) = 0 \right\rbrace$. The operator $A$  satisfies  (\ref{semigroupbound}) with $\omega = 0$ and $M=1$. We adjust the data $u_0$ and $f$ in such a way that $u(t,x) = x\text{e}^t$, $0 \leq t, x \leq 1$, is the solution of the problem (\ref{Example3}). According to the results in Section 2, it is straightforward to prove that $u(t, \cdot) \in X_{\nu}$, $0 \leq t \leq 1$, for every $0 < \nu < 1.5$.
   
We discretize (\ref{Example3}) by the method of lines, combining upwind finite diference for the discretization in space and the 3-stage SDIRK method ($p=4$, $q=1$) for the integration in time  (see, e.g., \cite{Hairer}). This method suffers from order reduction, and according to the main result in \cite{IC}, the reduced order turns out to be $p^* = q + \nu +1 = 3.5$. The method is implemented with the same values of $\mathbf{c_n}$ as in (\ref{dotgrid}), leading to the results shown in Table \ref{table5}.


\begin{table}[ht]
  \centering
  \caption{Orders for the hyperbolic example solved with SDIRK3 method (h=100). } 
    \begin{tabular}{rlrrrrr}
    \hline
    \multicolumn{1}{l}{Method} & Version  & \multicolumn{1}{l}{$\tau = 1/160$} & \multicolumn{1}{l}{$ 1/240$} & \multicolumn{1}{l}{$ 1/320$} & \multicolumn{1}{l}{$ 1/400$} & \multicolumn{1}{l}{$ 1/480$} \\
    \hline
    \multicolumn{1}{l}{SDIRK3} & Rational   & 3.97  & 3.98  & 3.99  & 3.99  & 3.99 \\
          & RK    & 2.89  & 3.17  & 3.34  & 3.45  & 3.52 \\
    \hline
    \end{tabular}%
    \label{table5}%
\end{table}%

\paragraph*{One-dimensional parabolic problem} Then, we introduce the one-dimensional heat equation with homogeneous boundary conditions
\begin{equation}
\label{Example1}
\left\{ \begin{array}{lcll} u_t(t,x) & = & u_{xx}(t,x) + f(t,x), & \quad 0\le t \le 1, \quad 0 \leq x \leq 1, \\
u(0,x)  & = & u_0(x), & \quad 0 \leq x \leq 1, \\
u(t,0) & = & 0, & \quad 0 \leq t \leq 1, \\
u(t,1) & = & 0, & \quad 0 \leq t \leq 1,
\end{array} \right.
\end{equation}
where $f : [0,1] \times [0,1] \rightarrow \mathbb{C}$, $u_0 : [0,1] \rightarrow \mathbb{C}$. In this case we consider $X = L^2[0,1]$, $A = d^2/dx^2$, $D(A) = H^2[0,1] \cap H^1_0[0,1]$. It is known that under these considerations the operator $A$ satisfies (\ref{semigroupbound}) with $\omega = 0$ and $M=1$. We adjust the data $u_0$ and $f$ in such a way that $u(t,x) = (1-x)\sin(tx)\text{e}^{t^2x}$, $0 \leq t \leq 1$, $0\leq x\leq 1$, is the solution of problem (\ref{Example1}). In this case, the results in Section 2 prove that $u(t, \cdot) \in X_{\nu}$, $0 \leq t \leq 1$, for every $0 < \nu < 1.25$.

The problem (\ref{Example1}) is discretized combining centered finite diference for the discretization in space and either the 3-stages Gauss method ($p=6$, $q = 3$) or the 3-stage SDIRK method ($p=4$, $q=1$) for the integration in time. According to \cite{IC}, the reduced orders are $p^* = q + \nu +1 = 5.25$ for the Gauss3 and $p^* = q + \nu +1 = 3.25$ for the SDIRK3. Table \ref{table1} shows the numerical orders obtained, which are in good agreement with the expected orders.

\begin{table}[ht]
  \centering
  \caption{Orders for the one-dimensional parabolic example (h=100). The symbol * stands for values of $\tau$ when the maximun precission is achieved.} 
    \begin{tabular}{rlrrrrr}
    \hline
    \multicolumn{1}{l}{Method} & Version & \multicolumn{1}{l}{$\tau = 1/20$} & \multicolumn{1}{l}{$1/40$} & \multicolumn{1}{l}{$ 1/80$} & \multicolumn{1}{l}{$1/160$} & \multicolumn{1}{l}{$ 1/320$} \\
    \hline
    \multicolumn{1}{l}{Gauss3} & Rational   & 5.52  & 5.85  & 5.83  & 5.96  & 5.98 \\
          & RK    & 4.99  & 5.14  & 5.20   & 5.14  & * \\
    \multicolumn{1}{l}{SDIRK3} & Rational   & 3.73  & 3.87  & 3.90   & 3.91  & 3.92 \\
          & RK    & 2.49  & 2.67 & 2.89  & 3.07   & 3.23 \\
    \hline
    \end{tabular}%
  \label{table1}%
\end{table}%

\paragraph*{Two-dimensional parabolic problem} Finally, we study the two dimensional problem in the square domain $\Omega = (0,1) \times (0,1)$, again with homogeneous Dirichlet boundary conditions,
\begin{equation}
\label{Example2}
\left\{ \begin{array}{lcll} u_t(t,x,y) & = & \Delta u(t,x,y) + f(t,x,y), & \quad 0\le t \le 1, \quad (x,y) \in \Omega, \\
u(0,x,y)  & = & u_0(x,y), & \quad (x,y) \in \Omega, \\
u(t,x,y) & = & 0, & \quad (x,y) \in \partial\Omega, \\
\end{array} \right.
\end{equation}
where $f : [0,1] \times \bar{\Omega} \rightarrow \mathbb{C}$, $u_0 : \bar{\Omega} \rightarrow \mathbb{C}$. We consider $X = L^2\left(\Omega\right)$, $A = \Delta$, $D(A) = H^2\left(\Omega\right) \cap H^1_0\left(\Omega\right)$. The bound (\ref{semigroupbound}) holds with $\omega = 0$ and $M=1$. We adjust the data $u_0$ and $f$ in such a way that $u(t,x,y) = x^3 y (x-1) (y-1)^3 \text{e}^t$, $0 \leq t \leq 1$, $(x,y) \in \bar{\Omega}$, is the solution of the problem (\ref{Example2}).

By the same reasoning, the main result in \cite{Lofs} guarantees that $u(t, \cdot,\cdot) \in X_{\nu}$, $0 \leq t \leq 1$, for every $0 < \nu < 1.25$, so the expected orders of convergence of the Gauss3 and SDIRK3 methods are the same as in the previous one-dimensional case. The method is implemented with the same values of $\mathbf{c_n}$ as in (\ref{dotgrid}).

\begin{table}[ht]
  \centering
  \caption{Orders for the 2D-parabolic example solved with Gauss3 method (h=100). The symbol * stands for values of $\tau$ when the maximun precission is achieved. } 
    \begin{tabular}{rlrrrrr}
    \hline
    \multicolumn{1}{l}{Method} & Version & \multicolumn{1}{l}{$\tau = 1/30$} & \multicolumn{1}{l}{$ 1/45$} & \multicolumn{1}{l}{$ 1/60$} & \multicolumn{1}{l}{$ 1/75$} & \multicolumn{1}{l}{$ 1/90$} \\
    \hline
    \multicolumn{1}{l}{Gauss3} & Rational   & 6.02  & 6.14  & 6.14  & 6.08  & * \\
          & RK    & 5.16  & 5.25  & 5.32  & 5.31  & 5.23 \\
    \hline
    \end{tabular}%
    \label{table2}%
\end{table}%

\begin{table}[ht]
  \centering
  \caption{Orders for the 2D-parabolic example solved with SDIRK3 method (h=100). }  
    \begin{tabular}{rlrrrrr}
    \hline
    \multicolumn{1}{l}{Method} & Version & \multicolumn{1}{l}{$\tau = 1/40$} & \multicolumn{1}{l}{$1/80$} & \multicolumn{1}{l}{$1/160$} & \multicolumn{1}{l}{$1/320$} & \multicolumn{1}{l}{$1/640$} \\
    \hline
    \multicolumn{1}{l}{SDIRK3} & Rational   & 4.02  & 4.06  & 4.01  & 3.94  & 3.97 \\
          & RK    & 2.55  & 2.76  & 2.99  & 3.17  & 3.27 \\
    \hline
    \end{tabular}%
    \label{table3}%
\end{table}%

\end{document}